\documentclass[final,3p,mathptmx]{elsarticle}%
\usepackage{eurosym}
\usepackage{amsfonts}
\usepackage{amsmath}
\usepackage{amssymb}
\usepackage{graphicx}
\usepackage{tikz}
\usepackage{color}
\usepackage{MnSymbol}
\usepackage{vmargin}
\usepackage{subfigure}
\usepackage{epsfig}
\usepackage{pgfplots}%
\providecommand{\U}[1]{\protect\rule{.1in}{.1in}}
\usetikzlibrary{shapes,snakes}
\usetikzlibrary{shapes.geometric}
\usetikzlibrary{calc,shadows}
\pgfplotsset{compat = newest}
\usetikzlibrary{shapes}
\usetikzlibrary{plotmarks}
\usetikzlibrary{calc}
\newtheorem{thm}{Theorem}
\newtheorem{lem}[thm]{Lemma}
\newtheorem{prop}[thm]{Proposition}
\newtheorem{rmk}{Remark}

\newtheorem{Cor}{Corollary}

\newenvironment{proof}[1][Proof] {\noindent\textbf{#1.} } {\ }

\newcommand{\Ss}{\mathcal{S}^{1}\left(\varphi,\, \overline{\tau}_{n}\right)}

\newcommand{\re}[1]{\color{black}#1}

\begin{document}
\begin{frontmatter}
\title{Normalized B-spline-like representation for low-degree Hermite osculatory interpolation problems}
\author[enes]{M. Boushabi}
\ead{mboushabi1@gmail.com}
\author[roma]{S. Eddargani\corref{dmb}}
\ead{eddargani@mat.uniroma2.it}
\cortext[dmb]{corresponding author.}
\author[granada]{M. J. Ib\'{a}\~{n}ez}
\ead{mibanez@ugr.es}
\author[enes]{A. Lamnii}
\ead{a.lamnii@uae.ac.ma}
\address[enes]{Abdelmalek Essaadi University, LaSAD, ENS, 93030 Tetouan, Morocco}
\address[roma]{Department of Mathematics, University of Rome Tor Vergata, 00133 Rome, Italy}
\address[granada]{Department of Applied Mathematics, University of Granada, 18071 Granada, Spain}

\begin{abstract}
This paper deals with Hermite osculatory interpolating splines. For a partition of a real interval endowed with a refinement consisting in dividing each subinterval into two small subintervals, we consider a space of smooth splines with additional smoothness at the vertices of the initial partition, and of the lowest possible degree. A normalized B-spline-like representation for the considered spline space is provided. In addition, several quasi-interpolation operators based on blossoming and control polynomials have also been developed. Some numerical tests are presented and compared with some recent works to illustrate the performance of the proposed approach.
\end{abstract}
\end{frontmatter}
\section{Introduction}
A spline function can be defined as a piecewise polynomial function whose pieces are connected together meeting given smoothness conditions. Spline functions are widely used in many areas, among them the computer-aided geometric design industry, where splines are used to represent geometric entities. The use of smooth splines of low degree is advantageous because {\color{black} it combines smoothness with a low degree that helps to effectively address a variety of problems.}

Given a partition $\tau_{n}:=\left\lbrace v_i,\, 0\leq i \leq n \right\rbrace $ of a bounded interval $I:=\left[a, b \right]$ into subintervals $I_i = \left[ v_i,\, v_{i+1} \right] $, $0\leq i \leq n-1$, and real values $f_i^j$, $0\leq i \leq n$, $0\leq j \leq \varphi(i)-1$, where $\varphi$ is a function that maps $\mathbb{N}$ into $\mathbb{N}$ with $\varphi(i)$  denoting the number of data imposed at  vertex $v_i$. The osculatory Hermite interpolation problem consists in finding a spline function $s$ such that (a) for any $0\leq i \leq n-1$, its restriction $s_i$ to $I_i$ is a polynomial of degree $1+\varphi(i) + \varphi(i+1)$;  (b) $s$ is $\mathcal{C}^{\varphi(i)-1}$ at $v_i$, $0 < i < n$; and (c) $s^{(j)} \left( v_i \right)= f_i^j$, $0\leq i \leq n$, $0\leq j \leq \varphi(i)-1$. This problem has a unique solution \cite{Lamnii2016}. Furthermore, the spline solution of this osculatory interpolation problem belongs to the linear space defined as follows: 
\[
\mathcal{S}\left( \varphi, \, \tau_n\right)=\left\lbrace s\in \mathcal{C}^{\varphi (i)-1} \left( v_i \right), 0< i< n \text{  and  } \, s_{\mid I_i} \in \mathbb{P}_{1+\varphi(i) + \varphi(i+1)}, \, 0\leq i \leq n-1 \right\rbrace ,
\]
where $\mathbb{P}_d$ denotes the linear space of polynomials of degree $\leq d$.

This problem has recently been addressed in \cite{Lamnii2016}, where a  basis based on two-point osculatory interpolation has been presented. These basis functions can be used as a useful tool for constructing curves, compressing data, and shape-preserving data. However, the elements of the proposed basis are not all non-negative as they define classical Hermite basis. Moreover, they are polynomials of degree $1+\varphi(i) + \varphi(i+1)$ on each subinterval $I_i$, $i=0, \ldots, n-1$.  In general, this kind of problem does not receive enough attention in the literature, a few papers dealing with osculatory interpolation can be found in the literature \cite{Int1,Int2,Int3,Int4}.

As pointed out in \cite{Lamnii2016}, the classical construction of Hermite osculatory interpolating splines requires that the spline degree in a subinterval $I_i$ must be $1+\varphi(i) + \varphi(i+1)$. From this insight, we wonder whether we can use a lower degree. Inspired by bivariate splines, we can consider the split of the initial partition by inserting new split points, i.e., dividing each sub-interval into some small sub-intervals. 

A refinement of the initial partition ensures the use of lower degrees in each element of the refined partition. This refinement approach that consists in dividing each initial subinterval into two small subintervals by adding a split knot is first considered in \cite{Schumaker1983}. It represents the univariate case of the Powell-Sabin 6-split \cite{Speleers2013, Jcam2022, Jcam2023}. The bivariate 6-split approach was introduced in \cite{PS} to construct $\mathcal{C}^1$ quadratic splines, resulting further in intensive research \cite{Dierckx,LaiAndSchumaker, Camwa,Lamnii2D,Sbibih2D,SpeleersG}

In this work, we consider a refinement $\overline{\tau}_n$ of the partition $\tau_n$, which is defined by inserting a split point $\zeta_i$ in each sub-interval $I_i$.  Let $\mathbb{P}_{\varphi, \overline{\tau}_n}$ be the space of piecewise polynomials of degree $\varphi (i)$ restricted to the sub-intervals sharing $v_i$ ($[\zeta_{i-1},\, v_i]$ and $[v_i,\, \zeta_{i}]$, $i=0, \ldots, n$, with $\zeta_{-1}=v_0,\,\zeta_n=v_n$) i.e. $\mathbb{P}_{\varphi, \overline{\tau}_n} = \left\lbrace p_{\mid [\zeta_{i-1},\, v_i]} ,\, p_{\mid [v_i,\, \zeta_{i}]}\in \mathbb{P}_{\varphi (i)}    \right\rbrace $. With $\zeta = \left\lbrace \zeta_i,\, i=0,\ldots , n-1 \right\rbrace$, an osculatory Hermite interpolating spline problem in the spline space 
\[
 \mathcal{S}^{1}\left(\varphi,\, \overline{\tau}_{n}\right)  :=\left\{  s\in  \mathcal{C}^{\varphi (i)-1} \left( v_i \right),\, s\in \mathcal{C}^{1}(\zeta)  :s\in \mathbb{P}_{\varphi, \overline{\tau}_n},\ 0< i <
n\right\},
\]
on the refined partition $ \overline{\tau}_{n} = \tau_n \cup \zeta$, will be stated and analysed. 

Functions in this spline space are $\mathcal{C}^1$ everywhere and $\mathcal{C}^{\varphi (i)-1}$  super smooth at the interior vertices $v_i$ of $ \tau_n$. The degree of the spline becomes $\varphi (i)$ on the sub-intervals $[\zeta_{i-1}, v_i]$ and $[v_i, \zeta_{i}]$.


 A local construction will be used to determine explicitly the solution of the osculatory interpolation problem. The spline space is characterized by an interpolation problem, therefore a basis is obtained as a dual of the basis of the dual space given by the interpolation functionals. In detail, each spline is uniquely determined by its values and those of its derivatives up to order $\varphi(i)-1$ at the vertices $v_i$, $i=0, \ldots, n$.

In contrast to recent works addressing the Hermite osculatory interpolation problem, where the basis used does not necessarily benefit from the usually required properties, such as the non-negativity of the basis functions and the partition of unity, and inspired by the construction in the bivariate case, we will provide a geometric approach leading us to construct a normalised B-spline-like basis for the spline space $\Ss$.

Furthermore, family of smooth quasi-interpolation schemes involving values and/or derivatives of a given function are also provided. The construction of these schemes is mainly based on establishing  Marsden's identity.  It is a potent tool that allows the monomials to be written in terms of the corresponding B-spline-like functions. To this end, we have established a general Marsden's identity in the spline space $\Ss$ using an approach based on the theory of control polynomials. The terminology and application of control polynomials for spline representations has been introduced in \cite{C1} and further developed in \cite{SpeleersG,C2}. It forms the foundation of many spline constructions, see, for example \cite{Jcam2022}.

The use of the refined partition offers some advantages. Specifically, it can be used to incorporate shape parameters in order to construct approximation splines preserving convexity or monotonicity to the given data \cite{Schumaker1983}. By imposing conditions on the position of the newly inserted points, it is possible to preserve the shape properties of the spline spaces proposed in this work. In this way, simpler results can be achieved than those available when using spaces without refinement. As a result, the normalized B-spline-like functions proposed here could be successfully used to define shape-preserving interpolating splines.

The considered spline space $\mathcal{S}^{1}\left(\varphi,\, \overline{\tau}_{n}\right) $ can be seen also as multi-degree spline space (in short MD-splines), which are an extension of standard B-splines. The term multi-degree refers to the concept of using B-splines with varying degrees to represent different segments of a curve or surface. These MD-splines were treated in different ways since their first introduction in \cite{MDB1}. In fact, the authors in \cite{MDB2} presented an approach for constructing MDB-splines by means of an integral recurrence relation. This recurrence construction is difficult to compute and instead of it, the authors in \cite{MDB3}  introduced a geometrically oriented recursive procedure. {\re Both works do} not allow the use of multiple knots to reduce the order of smoothness. Namely, both constructions yield $C^r$ continuous splines at the join between two pieces of the same degree $d=r+1$, and the smoothness order between two pieces of degrees $d_i$ and $d_{i+1}$ is $C^{\min \left( d_i, d_{i+1}\right) } $. A construction that extends the previous ones and that allows the use of multiple knots and, therefore, allows any order of continuity from $C^0$ to $C^{\min \left( d_i, d_{i+1}\right) } $, has been introduced in \cite{MDB4b} and further analysed in \cite{MDB4,MDB5}.  Furthermore, it has been investigated whether it is possible to have evaluation techniques based on algebraic recurrence relations, drawing on the famous "Cox-de Boor" method. But, recurrence schemes of this type have been identified and it has been proved that they only exist for certain MD-spline spaces, namely those in which the pieces of different degrees are glued together with continuity at the most $C^1$ \cite{MDB3}. Starting from this limitation, two approaches to deal with MD-spline spaces of arbitrary structure have been introduced. These two methods are based on a common basic idea, which is to map a set of known functions into the basis of interest. The first approach consists in computing the basis functions by interpolation, taking into account the fact that Hermite interpolation problems are unisolvent in MDB-spline spaces \cite{MDB4b,MDB4}. The second one comes from the idea of calculating explicitly, without solving any linear system as in the first approach. Namely, it defines a matrix operator which expresses the mapping between a well-known basis (Bernstein basis) and the set of MDB-splines \cite{MDB5,MDB6}. A Matlab Toolbox is provided in \cite{MDB7} to illustrate the computation and the use of MDB-splines. Recent work extends the theoretical and algorithmic approach of MDB-splines to generalized Tchebycheffian splines  \cite{MDB8,S22}. These non-polynomial splines exhibit favourable performance within the framework of Isogeometric analysis \cite{RCH23}, demonstrating their efficacy in avoiding oscillations by choosing the appropriate tension parameters, thereby ensuring stability and yielding robust numerical solutions.

As stated before, $\mathcal{S}^{1}\left(\varphi,\, \overline{\tau}_{n}\right) $ is a MD-spline space defined in a special structure. More precisely, when the degree changes from one piece to another, the smoothness at the shared knot is $C^1$, while the smoothness at a knot connecting two pieces of the same degree $d$ is of the order $C^{d-1}$. To construct a B-spline-like basis for $\mathcal{S}^{1}\left(\varphi,\, \overline{\tau}_{n}\right) $, we follow insights from the two approaches listed above. The term "like" means that the constructed spline basis has the same properties as classical B-splines. In particular, starting from the first approach presented in \cite{MDB4b}, we pretend that a B-spline-like basis can be determined by interpolation, avoiding at the same time the drawback of solving linear systems. And from the second approach, we note that a basis can be expressed locally in terms of easily computable basis, in particular Bernstein basis functions, and then glued together in a geometric intuition inspired by bivariate splines on micro-triangles. Therefore, we started with local  Bernstein bases relative to the breakpoint intervals, then we glued all the pieces together and the resulting B-spline-like basis is then also in Bernstein form.

The resulting B-spline-like basis functions of $\mathcal{S}^{1}\left(\varphi,\, \overline{\tau}_{n}\right) $ have the same useful properties as MDB-spline bases. They are non-negative, form a partition of unity and have local support. In terms of support, MDB-splines have the minimum possible support, which means that the support of the proposed B-spline-like basis is larger by a subinterval of that of MDB-splines. This occurs because here we associate to each initial vertex a set of basis functions, while we do not do so with the splitting points, which results in the fact that the support of the basis functions with respect to the initial vertices must be overlapped with the neighbours of the splitting points.

The remainder of this paper is organized as follows. Section $2$ is devoted to recalling some basic preliminaries about Bernstein-B\'{e}zier representation and some results on the blossom. In section $3$, we start from a classic Hermite basis for the space $\Ss$, then we provide a normalized B-spline-like basis using a geometric approach. In section $4$, we develop a family of quasi-interpolation schemes based on Marsden's identity. Finally,  we provide some numerical tests to illustrate the performance of the proposed quasi-interpolants in section $5$.
\section{Preliminaries}\label{pre}
In this section, we recall some results related to the Bernstein-B\'{e}zier (BB-) representation of polynomials and blossoming.

Consider the bounded interval $I=[a, b]$, the value of Bernstein basis polynomial $\mathfrak{B}_{\alpha,I}^{d}$ with the non-negative multi-index $\alpha:=\left(  \alpha_{1},\alpha_{2}\right)  $, $\left\vert \alpha\right\vert :=\alpha
_{1}+\alpha_{2}=d$, at any point $v \in I$ is given by
\[
\mathfrak{B}_{\alpha,\,I}^{d}\left(v \right):=\frac{d!}{\alpha_1! \alpha_2!} \frac{\left(b-v\right)^{\alpha_1}\left(v-a\right)^{\alpha_2}}{\left(b-a\right)^{d}},
\]
The set $\left\lbrace \mathfrak{B}_{\alpha,\,I}^{d}, \, \left\vert \alpha\right\vert =d \right\rbrace $ is a basis of the linear space $\mathbb{P}_{d}(I)$ of polynomials of degree less than or equal to $d$ defined over $I$, which is the best basis of $\mathbb{P}_{d}(I)$ according to the concept of optimal normalized totally positive basis \cite{Mazure2004,Carnicer1994}.

The BB-representation of a polynomial $p \in \mathbb{P}_{d}\left(I\right)$ is the unique linear combination
\begin{equation}\label{B_representation}
 p\left(v\right)=\sum_{|\alpha|=d}\,c_{\alpha}\, \mathfrak{B}_{\alpha,I}^{d}\left(v\right) ,
\end{equation}
where the coefficients $c_{\alpha}$ are said to be the B\'{e}zier (B-) ordinates of $p$. They are naturally linked to the domain points $\tfrac{1}{d} \left( \alpha_1 a +\alpha_2 b \right)$. These B-ordinates can be expressed as values of the blossom of $p$ (see \cite{Ramshaw}), i.e. the unique, symmetric, multi-affine polynomial  $\mathbf{B}\left[  p\right]$: $\left( \mathbb{R} \right)^d \rightarrow \mathbb{R} $ fulfilling the diagonal property $\mathbf{B}\left[  p\right] \left( v\left[ d\right]  \right) = p\left( v\right)$, where $v\left[ \ell\right]$ marks the repetition $\ell$ times of the point $v$ as an argument of the blossoming, omitting the term $\left[ \ell\right]$ when $\ell=1$. It holds,
\[
c_{\alpha}=\mathbf{B}\left[  p\right]  \left(  a[\alpha_{1}],\,b[\alpha_{2}]\right).
\]

In this work, we deal with osculatory interpolation, which means that the spline can change its degree from one subinterval to another. Establishing the appropriate smoothness conditions of a spline at the point shared by two subintervals is necessary. More precisely, let $I_{i,1}:=[v_{i},\zeta_{i}]$ and $I_{i,2}:=[\zeta_{i},v_{i+1}]$, and suppose that the restrictions $p_1$ and $p_2$ of a spline $s$ to $I_{i,1}$ and $I_{i,2}$ are in $\mathbb{P}_{d_1}$ and $\mathbb{P}_{d_2}$, respectively. The following result holds.
 \begin{lem}\label{lemma1}
The spline $s$ is $\mathcal{C}^{1}$ smooth at $\zeta_{i}$ if and only if 
\begin{equation}\label{C1_smoothness}
\hat{c}_{\left(d_2,0\right)}=c_{\left(0,d_1\right)}=\frac{1}{{\frac{d_1}{\zeta_{i}-v_{i}} }+\frac{d_2}{v_{i+1}-\zeta_{i}}}\left(\frac{d_2}{v_{i+1}-\zeta_{i}}\,\hat{c}_{\left(d_2-1,1\right)}+\frac{d_1}{\zeta_{i}-v_{i}}\, c_{\left(1,d_1-1\right)}\right),
\end{equation}
where ${c}_{\alpha}$  and  $\hat{c}_{\alpha}$ are the B-ordinates of $\mathit{p}_{1}$ and $\mathit{p}_{2}$, respectively.
 \end{lem}
 \begin{proof}
The polynomials  $p_{1}$ and $p_{2}$ are expressed in BB-representation as
\begin{align*}
p_{1}(v)& =\sum_{|\alpha|=d_1}\,c_{\alpha}\,\mathfrak{B}_{\alpha
,\,I_{i,1}}^{d_1}\left(  v\right)   ,\quad v\in I_{i,1},\\
p_{2}(v)& =\sum_{|\beta|=d_2}\,\hat{c}_{\beta}\,\mathfrak{B}_{\beta
,\,I_{i,2}}^{d_2}\left(  v\right)   ,\quad v\in I_{i,2},
\end{align*}
then,  
\begin{align*}
p_{1}^{\prime} \left(\zeta_{i}\right) & =\frac{d_1}{\zeta_{i}-v_{i}} \left( c_{\left(0,d_1\right)} - c_{\left(1,d_1-1\right)}\right) \text{  and  }
p_{2}^{\prime} \left(\zeta_{i}\right)= \frac{d_2}{v_{i+1}-\zeta_{i}}\left( \hat{c}_{\left(d_2-1,1\right)} -  \hat{c}_{\left(d_2,0\right)} \right).
\end{align*}
Since $s$ is $\mathcal{C}^{1}$ smoothness at the split point $\zeta_{i}$ if and only if $p_{1}^{(j)} \left(\zeta_{i}\right)=p_{2}^{(j)} \left(\zeta_{i}\right)$, $j=0, 1$. Then the claim follows.\qed
\end{proof}

The result in Lemma \ref{lemma1} can also be seen as a direct consequence of the $C^1$ imposition on MD-splines elaborated in Section 2.2 of \cite{C3}, taking into account:
\[
\alpha^{(i)} = \frac{d_1}{\zeta_{i}-v_i},\qquad \beta^{(i)} = \frac{d_2}{v_{i+1}-\zeta_{i}}.
\]
The behaviour of a spline function at any vertex can be predicted from the behaviour of the control polynomials at the same vertex. These polynomials are the main tool for establishing the Marsden's identity, which is the key to constructing spline quasi-interpolation schemes.

In what follows, we derive a result that leads us to define the control polynomials from the relationship between the polynomials and their blossoms. First, let us recall the following result \cite{Lamnii2014, Eddargani2021}.
\begin{lem}
\label{lemma_pp1} Let $d_{1}$ and $d_{2}$ be two positive integers ($d_{2} <  d_{1}$). For any polynomial $p\in\,\mathbb{P}_{d_{1}}$ and any set of values $x_{1},\ldots,\,x_{d_{1}-d_{2}}$ in $\mathbb{R}$, define,
\[
q\left(  x\right)  :=\mathbf{B}\left[  p\right]  \left(  x_{1},\ldots
,\,x_{d_{1}-d_{2}},\,x [d_2]\right). %
\]
The function $q$ is a polynomial of degree less than or equal to $d_{2}$, and for any
set of values $y_{1},\ldots,\,y_{d_{2}}$ in $\mathbb{R}$, it holds
\[
\mathbf{B}\left[  q\right]  \left(  y_{1},\ldots,\,y_{d_{2}}\right)
=\mathbf{B}\left[  p\right]  (x_{1},\ldots,\,x_{d_{1}-d_{2}},\,y_{1}%
,\ldots,\,y_{d_{2}}).
\]

\end{lem}

In what follows, we will define the control polynomial of degree $d_2$ at the vertex $v_{1}$ of a polynomial $p$ of degree $d_1$ ($d_2 \leq d_1$), which represents an alternative way to establish Marsden’s identity \cite{Matcom2022}. The following result can be seen as a restricted version of Theorem 3 in \cite{SpeleersG} and Theorem 1 in \cite{C2}.

\begin{prop}
\label{prop1} Consider  two positive integers $d_{1}$ and $d_{2}$, with
$d_{2} < d_{1}$. Let $p\in\,\mathbb{P}_{d_{1}} $ and $v_{1}\in\mathbb{R}$.
For any real number $\theta \neq 0$, the polynomial $q$ of degree $d_{2}$ defined by
\begin{equation}
q(v):=\mathbf{B}[p](v_{1}[d_{1}-d_{2}],\,(\frac{1}{\theta} v+(1-\frac{1}{\theta})v_{1})[d_{2}]),
\label{pp_pr1}%
\end{equation}
meets
\[
 p^{(j)}\left(  v_{1}\right)  \,=\,\theta^{j}\,\frac{\binom{d_{1}%
}{j}}{\binom{d_{2}}{j}}\,q^{(j)}(v_{1})
\]
for all $0\leq j\leq d_{2}$.
\end{prop}

{\re When $\theta = \frac{d_2}{d_1}$, it signifies that both the value and the first derivative of functions $p$ and $q$ at $v_1$ coincide. This implies that $q$ serves as the control polynomial of degree $d_2$ at $v_1$ for the polynomial $p$.}  It holds,

\begin{equation}\label{eq_for_pos}
c_{(d_1-d_2+k, \ell)}=\tilde{c}_{(k, \ell)},  \quad \text{  for all  } k+\ell=d_2,
\end{equation}
where $c_{(k, \ell)}$, $k+\ell=d_1$, and $\tilde{c}_{(k, \ell)}$, $k+\ell=d_2$ represent the B-ordinates of $p$ and $q$, respectively.

\section{Normalized B-spline-like representation}
Let $\tau_n$ and $\overline{\tau}_n$ be the subsets defined in the introduction, which define a partition of  $I$ and a refinement of it, respectively. A spline $s \in \Ss$ can be considered as a solution to the following Hermite osculatory interpolation problem:

\begin{thm}
Given values $ f_{i,j},\ i=0, \ldots, n,\ j=0,\dots, \varphi(i)-1$, there exists a unique spline $s \in \Ss $ such that
\begin{equation}\label{main_problem_interpolation}
s^{(j)}(v_i)=f_{i,j} \quad 0 \leq  i \leq n,\ 0 \leq j \leq \varphi(i)-1.
\end{equation}
\end{thm}
\begin{figure}[!htp]
\centering
\begin{tikzpicture}[scale=1]
\draw[gray] (0,0) --(14,0);
\draw[fill] (0,0) circle (0.05cm);
\draw[fill] (1.5,0) circle (0.05cm);
\draw[fill] (4.5,0) circle (0.05cm);
\draw[] (7,0) circle (0.05cm);
\draw[fill] (9.5,0) circle (0.05cm);
\draw[fill] (12,0) circle (0.05cm);
\draw[fill] (14,0) circle (0.05cm);
\node at (0,0.5) {\small {$c_{\left(\varphi(i) ,0\right)}$}};
\node at (1.5,0.5) {\small {$c_{\left(\varphi(i)-1,1\right)}$}};
\node at (3,0.5){\small{\dots}} ;
\node at (4.5,0.5) {\small {$c_{\left(1,\varphi(i)-1\right)}$}};
\node at (7,0.5) {\small {
$c_{\left(0,\varphi(i)\right)}=\hat{c}_{\left(\varphi(i+1),0\right)}$}};
\node at (9.5,0.5) {\small {$\hat{c}_{\left(\varphi(i+1)-1,1\right)}$}};
\node at (10.8,0.5) {\small {\dots}};
 \node at (12.2,0.5) {\small {$\hat{c}_{\left(1,\varphi(i+1)-1\right)}$}};
\node at (13.8,0.5) {\small {$\hat{c}_{\left(0,\varphi(i+1)\right)}$}};
\node at (0,-0.25) {\small {$v_{i}$}};
\node at (7,-0.25) {\small {$\zeta_{i}$}};
\node at (14,-0.25) {\small {$ v_{i+1}$}};
\end{tikzpicture}
\caption{A schematic representation of the B-ordinates of $s$ relative to the subintervals $[v_{i},\zeta_{i}]$ and $[\zeta_{i},v_{i+1}]$.}
\label{graphical_represenation_of Thm1}
\end{figure}
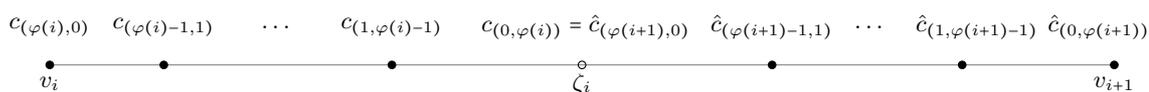
\begin{proof}
The proof will be done on a single interior sub-interval. Its extension to the whole partition can be carried out analogously. The B-ordinates of $s$ restricted to the interval $[v_{i},v_{i+1}]$ are schematically represented in Figure \ref{graphical_represenation_of Thm1}.

Since $s$ is $\mathcal{C}^{\varphi(i)-1}$ at $v_i$ and $\mathcal{C}^{\varphi(i+1)-1}$ at $v_{i+1}$, then the values and derivatives up to order $\varphi(i)-1$ at $v_i$ and $\varphi(i+1)-1$ at $v_{i+1}$, are uniquely determined by the B-ordinates relative to the domain points at distance less or equal to $\varphi(i)-1$ from $v_i$ and $\varphi(i+1)-1$ from $v_{i+1}$, i.e. the B-ordinates marked by ($\bullet$) in Figure \ref{graphical_represenation_of Thm1}.

\noindent The remaining B-ordinate, indicated by ($\circ$), is computed from $\mathcal{C}^1$ smoothness condition at $\zeta_i$.\qed
\end{proof}

In what follows, we will start by constructing a classical Hermite basis to $\Ss$, which is non-positive, and then we will provide a geometric approach that can help us to establish a new normalized B-spline-like basis.
\subsection{Classical Hermite basis}
Let $\phi_{i,j}$, $j=0, \ldots, \varphi(i)-1 $,  be the unique functions in $\Ss$ fulfilling the interpolation conditions
\begin{equation}
\phi_{i,j}^{(k)}(v_{\ell})=\delta_{i \ell}\delta_{j k},  \quad 0\leq i, \ell \leq n, \quad 0\leq j, k \leq \varphi(i)-1,
\end{equation}
where $\delta$ stands for the Kronecker delta. \\
Obviously, the functions $\phi_{i, j}$ , $j=0, \ldots, \varphi(i)-1 $, are locally supported on $\left[v_{i-1}, v_{i+1}\right]$, and every spline in $\Ss$ can be expressed as 
\begin{equation}
s=\sum_{i=0}^{n}\,\sum_{j=0}^{\varphi(i)-1} f_{i,j}\,\phi_{i, j}.
\end{equation}
The functions $\phi_{i, j}$ represent the classical Hermite basis of the space $\Ss $.


The fact that the elements of this basis are non-positive maybe results in reduced numerical stability. To overcome this drawback, we propose to construct a normalized basis of the spline space $\Ss $.
 
\subsection{Normalized B-spline-like basis}
We seek for linearly independent combinations of the classical Hermite basis functions that yield a B-spline-like basis. Namely, we will look for a suitable B-spline-like representation of $s \in \Ss$ as
\begin{equation}
s=\sum_{i=0}^{n}\,\sum_{|\alpha|=\varphi(i)-1}\,\mu_{i,\alpha}\,\mathcal{N}_{i,\alpha},\label{normalized-spline}
\end{equation}
in which the basis functions $\mathcal{N}_{i,\alpha}$ are non-negative, locally supported and form a partition of unity. We shall call $\mathcal{N}_{i,\alpha}$ a B-spline-like with respect to the vertex $v_i$.

In what follows, we will illustrate how to construct $\mathcal{N}_{i,\alpha}$, $i=0, \ldots, n$, $|\alpha|=\varphi(i)-1$. In fact, the construction used here has a geometric sense  and is inspired by the bivariate case \cite{Dierckx,C2}. It is entirely based on the choice of a single interval $S_i:= \left[ S_{i,1}, S_{i,2} \right] $ for each $v_i$ in the initial partition $\tau_n$. For each interior vertex $v_i$, define
\begin{equation}
S_{i,1}:=\theta \zeta_{i-1} +(1-\theta)v_{i} \text{  and  } S_{i,2}:=\theta \zeta_{i}+(1-\theta)v_{i}, \label{choiceofpoints}%
\end{equation}
where $0<\theta=\frac{\varphi(i)-1}{\varphi(i)}<1$. More generally, other values of $\theta$ could also be considered for practical purposes, even if it will not affect the shape of the basis functions $\mathcal{N}_{i,\alpha}$, see Section 3.3 of \cite{C2} for a detailed discussion.

Consider the Bernstein polynomials $\mathfrak{B}^{\varphi(i)-1}_{\alpha, S_i}$, $\vert \alpha \vert = \varphi(i)-1$, of degree $\varphi(i)-1$ on $S_{i}$, and define parameters
\begin{equation}
\gamma_{i,\alpha}^{j}:=\frac{%
\begin{pmatrix}
\varphi(i) \\
j
\end{pmatrix}
}{%
\begin{pmatrix}
\varphi(i)-1 \\
j
\end{pmatrix}
}\theta^{j}\,\left(\mathfrak{B}_{\alpha,\,S_{i}}^{\varphi(i)-1}\right)^{(j)}(v_{i}), \label{gamma_values}%
\end{equation}
for $0\leq j\leq \varphi(i)-1$, $|\alpha|=\varphi(i)-1$. They are used to define the B-spline-like functions $\mathcal{N}_{i,\alpha}$ as unique solutions of interpolation problems of the form (\ref{main_problem_interpolation}). More precisely,
\[
\mathcal{N}_{i,\alpha}^{(j)}(v_k)=\gamma_{i,\alpha}^{j}\delta_{ik}\delta_{j\alpha_2},\quad 0\leq i,k \leq n, \quad 0\leq j \leq \varphi(i)-1.
\]
The B-splines $\mathcal{N}_{i,\alpha}$, $\vert \alpha \vert=\varphi(i)-1$ are  linear combinations of the classical Hermite basis functions. Write
\[
\mathcal{N}_{i,\alpha} = \sum_{j=0}^{\varphi(i)-1} \gamma_{i,\alpha}^{j} \phi_{i,j}.
\]

\begin{rmk}\label{rmk_pos}
By the definition (\ref{gamma_values}) of the parameters providing the B-spline-like basis functions $\mathcal{N}_{i,\alpha} $, Proposition \ref{prop1} and (\ref{eq_for_pos}) show that the B-ordinates $c_{\varphi(i)-\ell, \ell}$,  $\ell=0, \ldots, \varphi(i)-1$, of $\mathcal{N}_{i,\alpha} $ relative to the domain points $\mathcal{D}_i:=\left\lbrace \frac{\varphi(i)-\ell}{\varphi(i)}v_i+\frac{\ell}{\varphi(i)}\zeta_{i},\right. $ $\left.  \, \ell=0, \ldots, \varphi(i)-1\right\rbrace$ can be seen as the B-ordinates of the Bernstein basis polynomial $\mathfrak{B}_{\alpha,\,S_{i}}^{\varphi(i)-1}$ of degree $\varphi(i)-1$, defined on $S_{i}$ (after subdivision). Moreover, they can be expressed in terms of blossom as follows,
\begin{equation}\label{eq_pos_2}
c_{\varphi(i)-\ell, \ell} = \mathbf{B}[\mathfrak{B}_{\alpha,\,S_{i}}^{\varphi(i)-1}]\left( \tau [\varphi(i)-\ell-1], (1-\tau) [\ell]  \right),
\end{equation}
where $\left( \tau, 1-\tau\right)$ represents the barycentric coordinates of the points of $\mathcal{D}_i$ with respect to the interval $S_i$.
\end{rmk}

As mentioned before, among the essential properties in many fields is that the elements of the basis of the spline space should be non-negative and form a convex partition of unity. These two properties are stated and proved in the following two results.
\begin{lem}
The B-spline-like functions $\mathcal{N}_{i,\alpha}$, $\vert \alpha \vert=\varphi(i)-1$, associated with the vertex $v_i \in \tau_n$ are non-negative.
\end{lem}
\begin{proof}

 For fixed $\alpha$ and $i$, it is sufficient to show that all B-ordinates of $\mathcal{N}_{i,\alpha}$ are non-negative.   Denote by $c_{\beta}$, $\beta \in \mathbb{N}^2$, $\vert \beta \vert=\varphi(i)$ the B-ordinates of $\mathcal{N}_{i,\alpha}$ restricted to the interval $[v_i, \zeta_i]$. From Remark \ref{rmk_pos}, the B-ordinates $c_{\varphi(i)-\ell, \ell}$, $\ell=0, \ldots, \varphi(i)-1$ can be expressed in the form (\ref{eq_pos_2}).  They are a blossom of Bernstein polynomials defined on $S_i$ with argument $(\tau, 1-\tau)$. This blossom is non-negative if the barycentric coordinates $(\tau, 1-\tau)$ are non-negative. This holds, if the points in $\mathcal{D}_i$ are all inside $S_i$.

By symmetry, the same condition can be imposed on the set of points $\left\lbrace \frac{\varphi(i)-\ell}{\varphi(i)}v_i+\frac{\ell}{\varphi(i)}\zeta_{i-1},\right. $ $\left.  \, \ell=0, \ldots, \varphi(i)-1\right\rbrace$, in order to ensure the non-negativity of all B-ordinates $c_{\beta}$ of $\mathcal{N}_{i,\alpha}$.   \qed
\end{proof}

\begin{lem}
The B-spline-like functions $\mathcal{N}_{i,\alpha}$, $i=0,\ldots,n$, $|\alpha|=\varphi(i)-1$ form a convex partition of unity, i.e.
\[
\sum_{i=0}^{n}\,\sum_{|\alpha|=\varphi(i)-1}\,\mathcal{N}_{i,\alpha}=1.
\]
\end{lem}

\begin{proof}
From the definition of the B-spline-like functions it follows that only $\varphi(i)$ basis functions have function and derivative values at the vertex $v_i$ that are not all zero.\\
 Moreover, the Bernstein basis polynomials in (\ref{gamma_values}) form a partition of unity on $S_i$. \\
Then {\re one can obtain} that
\begin{equation}
\sum_{|\alpha|=\varphi(i)-1}\,\gamma_{i,\alpha}^{0}=1,\quad \text{and}\quad\sum_{|\alpha|=\varphi(i)-1}%
\,\gamma_{i,\alpha}^{j}=0, \quad \quad 0\leq j\leq \varphi(i)-1 . \label{weight_gamma}%
\end{equation}
The proof is completed by considering interpolation problem (\ref{main_problem_interpolation}) and (\ref{weight_gamma}).\qed
\end{proof}

Figure \ref{figure_interior} shows two examples of B-spline like basis functions with respect to an interior vertex $v_i \in \tau_n$: (left) functions $\mathcal{N}_{i,\alpha}$, $\vert \alpha \vert = 2$, associated with the vertex $v_i$, where three data are provided, while two data are assumed to be provided at $v_{i-1}$ and four data are imposed at $v_{i+1}$ (i.e. $\varphi(i)=3, \varphi(i-1)=2, \varphi(i+1)=4$). (Right) functions $\mathcal{N}_{i,\alpha}$, $\vert \alpha \vert = 3$, corresponding to $v_i$, where two data are available at $v_{i-1}$, while seven data are available at $v_{i+1}$ (i.e. $\varphi(i)=4, \varphi(i-1)=2, \varphi(i+1)=7$).

\begin{figure}[!h]
\centering
\includegraphics[scale=0.8]{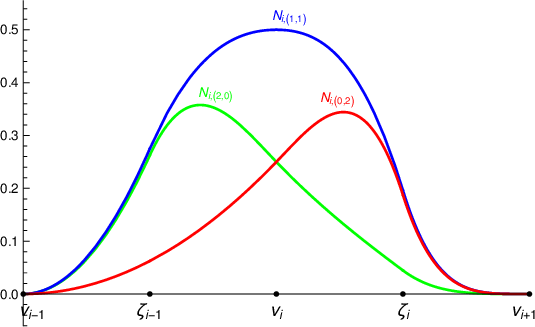}\,\includegraphics[scale=0.8]{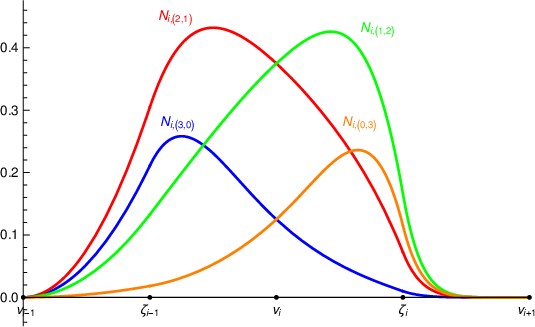}
\caption{Examples of B-spline-like functions associated with an interior vertex $v_i$: (left) The three basis functions corresponding to problem (\ref{main_problem_interpolation}) with $\varphi(i)=3$, $\varphi(i-1)=2$ and $\varphi(i+1)=4$. (right) The four basis functions corresponding to problem (\ref{main_problem_interpolation}) with $\varphi(i)=4, \varphi(i-1)=2, \varphi(i+1)=7$.}\label{figure_interior}
\end{figure}

\begin{rmk}
The construction of boundary B-spline-like basis functions is done according to the same strategy used for interior vertices. Namely, the B-spline-like functions with respect to the boundary vertex $v_0=a$ (resp. $v_n=b$) are constructed with a particular choice of the interval $S_0$ (resp. $S_n$). Namely, $S_{0,1} = v_0 \text{   and   } S_{n,2} = v_n$.
\end{rmk}

In Figure \ref{figure_boundary} we illustrate the typical plots of the boundary B-spline-like functions considered in Figure \ref{figure_interior}.

\begin{figure}[!h]
\centering
\includegraphics[scale=0.8]{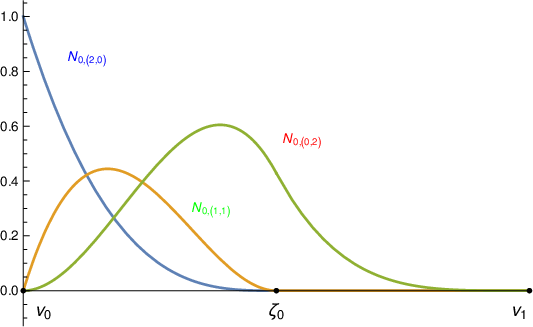} \quad \includegraphics[scale=0.8]{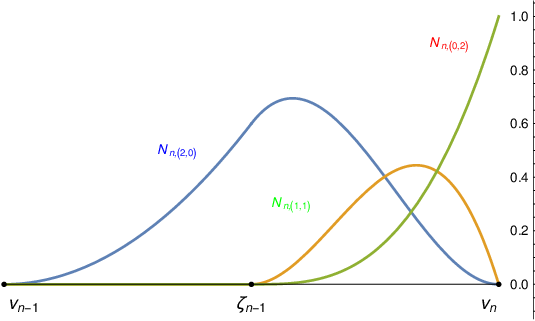}
\includegraphics[scale=0.8]{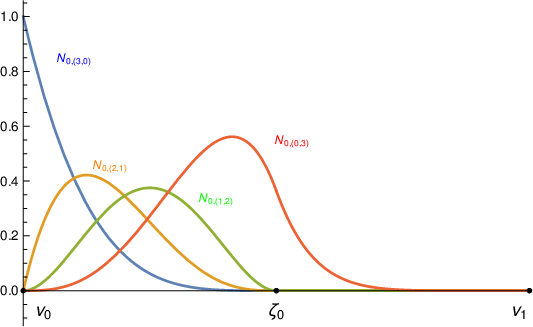} \quad \includegraphics[scale=0.8]{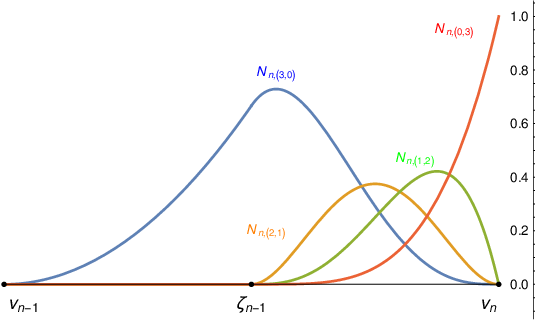}
\caption{Plots of the boundary B-spline-like functions outlined in Figure \ref{figure_interior}: (Top, left) the case of three data at $v_0$ and four data at $v_1$. (Top, right) two data at $v_{n-1}$ and three at $v_n$. (Bottom, left) the case of four data at $v_0$ and seven at $v_1$. (Bottom, right) two data at $v_{n-1}$ and four at $v_n$. }\label{figure_boundary}
\end{figure}
\subsection{B-spline-like representation}
In this subsection, we first start by deriving the coefficients of (\ref{normalized-spline}) to get an interpolating spline.
Suppose that $ s \in \Ss$ is determined by the interpolation problem (\ref{main_problem_interpolation}). Then, the evaluation of $s^{(j)}$, $0 \leq j \leq \varphi(i)-1$, at $v_i$ yields the linear system
\begin{equation}\label{linear_system}
\begin{pmatrix}
\gamma_{i,(\varphi(i)-1,0)}^{0} & \ldots & \gamma_{i,(\varphi(i)-1-\ell,\ell)}^{0} & \ldots & \gamma_{i,(0,\varphi(i)-1)}^{0}\\
\vdots &  & \vdots & &  \vdots\\
\gamma_{i,(\varphi(i)-1,0)}^{j} & \ldots & \gamma_{i,(\varphi(i)-1-\ell,\ell)}^{j} & \ldots & \gamma_{i,(0,\varphi(i)-1)}^{j}\\
\vdots &  & \vdots &  & \vdots\\
\gamma_{i,(\varphi(i)-1,0)}^{\varphi(i)-1} & \ldots & \gamma_{i,(\varphi(i)-1-\ell,\ell)}^{\varphi(i)-1} & \ldots & \gamma_{i,(0,\varphi(i)-1)}^{\varphi(i)-1}%
\end{pmatrix}%
\begin{pmatrix}
\mu_{i,(\varphi(i)-1,0)}\\
\\
\vdots\\
\\
\mu_{i,(0,\varphi(i)-1)}%
\end{pmatrix}
=%
\begin{pmatrix}
f_{i,0}\\
\\
\vdots\\
\\
f_{i,\varphi(i)-1}%
\end{pmatrix}
.
\end{equation}
The definition of parameters $\gamma_{i,\alpha}^{j}$ in (\ref{gamma_values}) includes values and derivatives of Bernstein polynomials. Since they are linear independent, the solution of the linear system is unique.

The polynomial function of degree $\varphi(i)-1$ defined  on $S_i$ with the coefficients $\mu_{i,\alpha}$ as B-ordinates is called the control polynomial related to the vertex $v_i$. It is denoted by $T_i(x)$, and it is expressed as 

\[
T_{i}(v):=\sum_{|\alpha|=\varphi(i)-1}\,\mu_{i,\alpha}\,\mathfrak{B}_{\alpha,\,S_{i}}^{\varphi(i)-1}(v),\quad v\in S_{i}.
\]

The following result holds.
\begin{thm}
 $T_{i}$ is tangent to the spline $s\in \Ss  $ at $v_{i}$.
\end{thm}
\begin{proof}
For $s\in \Ss $, it holds
\[
s^{(j)}\left(  v_{i}\right)  =\sum_{|\alpha|=\varphi(i)-1}\,\mu_{i,\alpha}\,\gamma
_{i,\alpha}^{j}=\sum_{|\alpha|=\varphi(i)-1}\,\mu_{i,\alpha}\,\left(\mathfrak{B}_{\alpha,\,S_{i}}%
^{\varphi(i)-1}\right)^{(j)} (v_i)=T_{i}^{(j)}(v_{i}),\quad j=0,1.
\]
and  the claim follows.\qed

{\re \begin{rmk}
As $\theta=\frac{\varphi(i)-1}{\varphi(i)}$, it is worthwhile to mention that the control polynomial $T_i$ coincides with the polynomial $q$ in Proposition \ref{prop1}, given that $p$ represents the restriction of the spline $s$ in a sub-interval containing $v_i$.
\end{rmk}}
\end{proof}
\section{Quasi-interpolants schemes}

In this section, we seek to build some quasi-interpolation operators that map an element of the linear space of piecewise polynomials  to an element of $\Ss$. 

Then, we have the following result. 
 \begin{prop} For any $p \in \mathbb{P}_{\varphi, \overline{\tau}_n}$, let $\mathcal{Q}_{\varphi, \overline{\tau}_n}$ be defined as 
\begin{equation}
\mathcal{Q}_{\varphi, \overline{\tau}_n}\,p=\sum_{i=0}^{n}\,\sum_{|\alpha|=\varphi(i)-1}\,\mathbf{B}\left[p_{\mid [ v_i, \zeta_i]}\right] \left(  v_{i},{\re\zeta_{i-1}[\alpha_{1}],\,\zeta_{i}[\alpha_{2}]}\right)
\,\mathcal{N}_{i,\alpha}. \label{quasi_interpolant}%
\end{equation}
Then, $\mathcal{Q}_{\varphi, \overline{\tau}_n}\,p\in \Ss  $ and
 $\mathcal{Q}_{\varphi, \overline{\tau}_n}\,\tilde{p}=\tilde{p}$ for all  $\tilde{p} \in \mathbb{P}_{\min_{0\leq i \leq n} \varphi (i)}$.
\end{prop}
\begin{proof}
Let 
\[
s(v)=\sum_{i=0}^{n}\,\sum_{|\alpha|=\varphi(i)-1}\,\mathbf{B}\left[  p_{\mid [ v_i, \zeta_i]}\right]
\left(  v_{i},{\re \zeta_{i-1}[\alpha_{1}],\,\zeta_{i}[\alpha_{2}]}\right)
\,\mathcal{N}_{i,\alpha}\left(v\right).
\]
 We will prove that,
\[
s^{(j)} (v_i) = p^{(j)} (v_i),\quad 0\leq j \leq \varphi(i)-1.
\]
It is clear that 
\[
s(v_i)=\sum_{|\alpha|=\varphi(i)-1}\,\mathbf{B}\left[  p_{\mid [ v_i, \zeta_i]}\right]
\left(  v_{i},{\re\zeta_{i-1}[\alpha_{1}],\,\zeta_{i}[\alpha_{2}]}\right)
\,\mathcal{N}_{i,\alpha}(v_i).
\]
Define 
\[
q_i(v)=\sum_{|\alpha|=\varphi(i)-1}\,\mathbf{B}\left[  p_{\mid [ v_i, \zeta_i]}\right]
\left(  v_{i},{\re\zeta_{i-1}[\alpha_{1}],\,\zeta_{i}[\alpha_{2}]}\right)
\,\mathcal{N}_{i,\alpha}(v).
\]
For all $0 \leq j \leq \varphi(i)-1$, it holds
\[
q^{(j)}_{i}(v_i)=\frac{%
\begin{pmatrix}
\varphi(i)\\
j
\end{pmatrix}
}{%
\begin{pmatrix}
\varphi(i)-1\\
j
\end{pmatrix}
}\,\theta^{j}\sum_{|\alpha|=\varphi(i)-1}\,\mathbf{B}\left[  p_{\mid [ v_i, \zeta_i]}\right]
\left(  v_{i},{\re\zeta_{i-1}[\alpha_{1}],\,\zeta_{i}[\alpha_{2}]}\right)
\,\left(\mathfrak{B}_{\alpha,\,S_{i}}^{\varphi(i)-1}\right)^{(j)} (v_i).
\]
Now, we consider the notion of control polynomial developed in the Section \ref{pre}.\\
Let 
\[
{\re T_{i}(v)}=\mathbf{B}\left[ p_{\mid [ v_i, \zeta_i]}\right]
\left( v_i,\left(\frac{1}{\theta}\, v+\frac{\theta-1}{\theta}\,v_i\right) [\varphi(i)-1]\right)
\]
be the control polynomial of degree $\varphi(i)-1$ at the vertex $v_i$. The polynomial ${\re T_{i}(v)}$ can be written on $[S_{i,1},S_{i,2}]$ as
\[
{\re T_{i}(v)}=\sum_{|\alpha|=\varphi(i)-1}\,\mathbf{B}\left[  T_{i}\right]
\left( S_{i,1}[\alpha_{1}],\,S_{i,2}[\alpha_{2}]\right)\,
\mathfrak{B}_{\alpha,\,S_{i}}^{\varphi(i)-1}(v).
\]
According to Lemma \ref{lemma_pp1}, we have
\[
{\re T_{i}(v)}=\sum_{|\alpha|=\varphi(i)-1}\,\mathbf{B}\left[  p_{\mid [ v_i, \zeta_i]}\right]
\left(v_i, {\re \left(\frac{1}{\theta}\,S_{i,1}+\left(\frac{\theta-1}{\theta}\right)v_{i}\right)\left[ \alpha_{1}\right] ,\,\left(\frac{1}{\theta}\,S_{i,2}+\left(\frac{\theta-1}{\theta}\right)v_{i}\right)\left[ \alpha_{2}\right] }\right)\,
\mathfrak{B}_{\alpha,\,S_{i}}^{\varphi(i)-1}(v).
\]
{\re A straightforward computation shows that   $\zeta_{i-1} = \frac{1}{\theta}\,S_{i,1}+\left(\frac{\theta-1}{\theta}\right)v_{i}$ and  $\zeta_{i} = \frac{1}{\theta}\,S_{i,2}+\left(\frac{\theta-1}{\theta}\right)v_{i}$ , which leads to the following,
\[
T_{i}(v)=\sum_{|\alpha|=\varphi(i)-1}\,\mathbf{B}\left[  p_{\mid [ v_i, \zeta_i]}\right]
\left(v_i, \zeta_{i-1}[\alpha_{1}],\,\zeta_{i}[\alpha_{2}]\right)\,
\mathfrak{B}_{\alpha,\,S_{i}}^{\varphi(i)-1}(v),
\]}
By using Proposition \ref{prop1}, we deduce that
\[
 p^{(j)}(v_{i})= \theta ^{j}\,\frac{\binom{\varphi(i)}{j}}{\binom
{\varphi(i)-1}{j}}T_{i}^{(j)}(v_i)=q_{i}^{(j)}(v_{i})=\,s^{(j)}(v_{i}),\quad 0\leq j \leq \varphi(i)-1
\]
which concludes the proof. \qed
\end{proof}

Next, we define from (\ref{quasi_interpolant}) a family of quasi-interpolation operators $\mathcal{Q}_{\varphi, \overline{\tau}_n}$ of the form
\begin{equation}
\mathcal{Q}_{\varphi, \overline{\tau}_n}f:=\sum_{i=0}^{n}\,\sum_{\vert \alpha \vert=\varphi\left(i\right)-1}\,\nu_{i,\alpha}\left(
f\right)  \,\mathcal{N}_{i,\alpha}. \label{form1}%
\end{equation}
The linear functionals $\nu_{i,\alpha}$ are defined such that $\mathcal{Q}_{\varphi, \overline{\tau}_n}$ meets the property:

\begin{equation}
 \mathcal{Q}_{\varphi, \overline{\tau}_n}\, \tilde{p}=\tilde{p} \text{  for all  } \tilde{p} \in \mathbb{P}_{\min_{0\leq i \leq n} \varphi (i)}.
 \label{form2}%
\end{equation}

\subsection{Differential quasi-interpolation operator}

We begin with the result that shows the relationship between blossoming and directional derivatives according to \cite{Seidel}.
\begin{prop}
Let $u$, $v$, $w$ be three points in $\mathbb{R}$ and $ p_i \in \mathbb{P}_{\varphi(i)},\,i=0,\dots,n-1 $.Then,
\begin{equation}\label{polarization_der}
 \mathbf{B}\left[  p_i\right]  \left(  u ,v\left[  \alpha_{1}\right] ,w\left[  \alpha_{2}\right] \right)
=\sum_{k=0}^{\varphi(i)-1} \frac{(k+1)!}{\varphi(i)!}\sum_{\alpha_1 + \alpha_2 = \varphi(i)-1-k }
\delta^{\alpha_1}\, \hat{\delta}^{\alpha_2}  p^{(\alpha_1+\alpha_2)}\left(u\right),
\end{equation}
where $0\leq \alpha_1, \alpha_2 \leq \varphi(i)-1$, $\alpha_1+\alpha_2=\varphi(i)-1$, $\delta:= v-u$ and $\hat{\delta}= w-u$.
\end{prop}
From the functional defined as 
\[
 \mathbf{N}\left[  f\right]  \left(  u ,v\left[  \alpha_{1}\right] ,w\left[  \alpha_{2}\right] \right)
=\sum_{k=0}^{\varphi(i)-1} \frac{(k+1)!}{\varphi(i)!}\sum_{\alpha_1 + \alpha_2 = \varphi(i)-1-k }
\delta^{\alpha_1}\, \hat{\delta}^{\alpha_2}  p^{(\alpha_1+\alpha_2)}\left(u\right),
\]
we define linear functionals providing differential quasi-interpolation operators.
\begin{Cor}
Define the functional
\begin{equation}
\nu_{i,\alpha}\left(  f\right)  :=\mathbf{N}\left[  f\right]
\left(  v_{i},{\re\zeta_{i-1}[\alpha_{1}],\,\zeta_{i}[\alpha_{2}]}\right) . \label{df_qi}%
\end{equation}
Then, the operator $\mathcal{Q}_{\varphi, \overline{\tau}_n}$ defined by (\ref{form1}) satisfies (\ref{form2}).

\end{Cor}

\begin{proof}
It is enough to notice that
\[
\mathbf{N}\left[  p_i\right]
\left(  v_{i},{\re\zeta_{i-1}[\alpha_{1}],\,\zeta_{i}[\alpha_{2}]}\right)  =\mathbf{B}\left[    p_i\right]
\left(  v_{i},{\re\zeta_{i-1}[\alpha_{1}],\,\zeta_{i}[\alpha_{2}]}\right),
\]
for all $ p_i =p_{\mid [v_{i},\zeta_i]}\in \mathbb{P}_{\varphi(i)},\,i=0,\dots,n-1$, and $p \in \mathbb{P}_{\varphi, \overline{\tau}_n}$.\qed

\end{proof}

\subsection{Quasi-interpolation based on point values}

The construction of differential quasi-interpolants requires derivative information and given data values. In general, such information on derivatives is only sometimes available in practice. For this reason, we construct quasi-interpolants based on point evaluators in this subsection.

Let $t_{i,\alpha}^{\ell}$, $\ell=0,\dots,\varphi(i)$, be $\varphi(i) +1 $ distinct points in the neighborhood of $v_i$. Then, there exist a Lagrange basis $L_{i,\alpha}^j$ of degree $\varphi(i)$ such that $L_{i,\alpha}^j \left( t_{i, \alpha}^{\ell} \right) =\delta_{j, \ell}$, $j, \ell=0,\dots,\varphi(i)$, and the polynomial 
\begin{equation}\label{QI_pv}
\mathcal{I}_{i, \alpha} f(x):=\sum_{\ell=0}^{\varphi\left(i\right)}\,f(t_{i, \alpha}^{\ell})\,\mathit{L}_{i, \alpha}^\ell,
\end{equation}
interpolates $f$ at the points $ t_{i,\alpha}^{\ell}$, $\ell=0,\ldots, \varphi(i)$.

Consequently, we have the following result.
\begin{prop}
The quasi-interpolant $\mathcal{Q}_{\varphi, \overline{\tau}_n} f$ in (\ref{form1}) with
\begin{equation}
\nu_{i,\alpha} (f) = \sum_{\ell=0}^{\varphi(i)}\, f\left( t_{i, \alpha}^{\ell}\right) \mathbf{B}\left[\mathit{L_{i, \alpha}^\ell}\right]\left(  v_{i},{\re\zeta_{i-1}[\alpha_{1}],\,\zeta_{i}[\alpha_{2}]}\right),
\end{equation}
satisfies (\ref{form2}).
\end{prop}
For instance, we give some examples of discrete quasi-interpolation operators based on point values for a uniform partition. Consider the case of $\varphi(2 i)=3$ and $\varphi(2 i+1)=4$. For each vertex $v_i$, we should define a polynomial $\mathcal{I}_{i, \alpha} = \mathcal{I}_{i}$ of degree $\varphi (i)$ which interpolates in the neighbourhood of $v_i$. In the general construction, we can define a polynomial $\mathcal{I}_{i, \alpha}$ of degree $\varphi(i)$ for each of the functional $\nu_{i,\alpha}$, but for the sake of simplicity, we choose the same interpolation polynomial $\mathcal{I}_{i}$ for the all functionals $\nu_{i, \alpha}$

\begin{itemize}
\item[1)] Case $\varphi(i)=3$: The polynomial $\mathcal{I}_{i} $  interpolates the points $v_{i-1},\,\zeta_{i-1},\,v_{i},\,\zeta_{i}$, i.e.,
\[
\mathcal{I}_{i} f(x) = f(v_{i-1}) L_{i}^1(x) + f(\zeta_{i-1}) L_{i}^2(x) + f(v_{i}) L_{i}^3(x) + f(\zeta_{i}) L_{i}^4(x),
\]
where,
\begin{align*}
L_{i}^1(x) =& \frac{(x-\zeta_{i-1})(x-v_{i})(x-\zeta_{i})}{(v_{i-1}-\zeta_{i-1})(v_{i-1}-v_{i})(v_{i-1}-\zeta_{i})},\quad L_{i}^2(x) = \frac{(x-v_{i-1})(x-v_{i})(x-\zeta_{i})}{(\zeta_{i-1}-v_{i-1})(\zeta_{i-1}-v_{i})(\zeta_{i-1}-\zeta_{i})}\\
L_{i}^3(x) =& \frac{(x-v_{i-1})(x-\zeta_{i-1})(x-\zeta_{i})}{(v_{i}-v_{i-1})(v_{i}-\zeta_{i-1})(v_{i}-\zeta_{i})},\quad L_{i}^4(x) = \frac{(x-v_{i-1})(x-\zeta_{i-1})(x-v_{i})}{(\zeta_{i}-v_{i-1})(\zeta_{i}-\zeta_{i-1})(\zeta_{i}-v_{i})}.
\end{align*}
Which yields,
\begin{align*}
\nu_{i,(2,0)}\left(  f\right)   &  =\frac{-2}{18}f\left(  v_{i-1}\right)+\frac{15}{18}f\left(  \zeta_{i-1}\right)+\frac{6}{18}f\left(  v_{i}\right)-\frac{1}{18}f\left(  \zeta_{i}\right)  ,\\
\nu_{i,(1,1)}\left(  f\right)   &  =\frac{-1}{6}f\left(  \zeta_{i-1}\right)+\frac{8}{6}f\left(  v_{i}\right)-\frac{1}{6}f\left(  \zeta_{i}\right)  ,\\
\nu_{i,(0,2)}\left(  f\right)   &  =\frac{1}{9}f\left(  v_{i-1}\right)-\frac{1}{2}f\left(  \zeta_{i-1}\right)+f\left(  v_{i}\right)+\frac{7}{18}f\left(  \zeta_{i}\right)  .
\end{align*}
Similarly, one can compute the functionals $\nu_{i,\alpha}$ in the case of $\varphi(i)=4$.
\item[2)] Case $\varphi(i)=4$: The interpolation points $v_{i-1},\,\zeta_{i-1},\,v_{i},\,\zeta_{i}, v_{i+1}$ are considered. It results,
\begin{align*}
\nu_{i, (3,0)}\left(  f\right)   &  =\frac{-3}{48}f\left(  v_{i-1}\right)+\frac{38}{48}f\left(  \zeta_{i-1}\right)+\frac{18}{48}f\left(  v_{i}\right)-\frac{6}{48}f\left(  \zeta_{i}\right)+\frac{1}{48}f\left(  v_{i+1}\right)  ,\\
\nu_{i, (2,1)}\left(  f\right)   &  =\frac{-5}{144}f\left(  v_{i-1}\right)+\frac{14}{144}f\left(  \zeta_{i-1}\right)+\frac{174}{144}f\left(  v_{i}\right)-\frac{46}{144}f\left(  \zeta_{i}\right)+\frac{7}{144}f\left(  v_{i+1}\right)   ,\\
\nu_{i, (1,2)}\left(  f\right)   &  =\frac{7}{144}f\left(  v_{i-1}\right)-\frac{46}{144}f\left(  \zeta_{i-1}\right)+\frac{174}{144}f\left(  v_{i}\right)+\frac{14}{144}f\left(  \zeta_{i}\right)-\frac{5}{144}f\left(  v_{i+1}\right)  ,\\
\nu_{i, (0,3)}\left(  f\right)   &  =\frac{1}{48}f\left(  v_{i-1}\right)-\frac{6}{48}f\left(  \zeta_{i-1}\right)+\frac{18}{48}f\left(  v_{i}\right)+\frac{38}{48}f\left(  \zeta_{i}\right)-\frac{3}{48}f\left(  v_{i+1}\right)  .
\end{align*}
\end{itemize}
The value of the polar forms in a set of points can be calculated using derivative formulas (\ref{polarization_der}), or even using a discrete polarisation formula, which we will recall in the next subsection.

\subsection{Discrete quasi-interpolation operator based on polarization}

Polarization formula with constant coefficients can be obtained as combinations of discrete values \cite{Ramshaw}. Let $p\in \mathbb{P}_n$, it holds
\[
\mathbf{B}\left[  p\right]  \left(  u_{1},\,\ldots,\,u_{n}\right)  =\frac
{1}{n!}\sum_{\underset{k=|S|}{S\subset\left\{  1,\ldots,\,n\right\}  }%
}\,(-1)^{n-k}\,k^{n}\,p\left(  \frac{1}{k}\sum_{i\in S}\,u_{i}\right)  .
\]
Define
\[
\mathbf{M}\left[  f\right]  \left(  u_{1},\,\ldots,\,u_{n}\right)  :=\frac
{1}{n!}\sum_{\underset{k=|S|}{S\subset\left\{  1,\ldots,\,n\right\}  }%
}\,(-1)^{n-k}\,k^{n}\,f\left(  \frac{1}{k}\sum_{i\in S}\,u_{i}\right)  .
\]
We have the following result.

\begin{Cor}
The quasi-interpolation operator $\mathcal{Q}_{\varphi, \overline{\tau}_n}$ defined by (\ref{form1})
with
\begin{equation}
\nu_{i,\alpha}\left(  f\right)  =\mathbf{M}\left[  f\right]
\left(  v_{i},{\re\zeta_{i-1}[\alpha_{1}],\,\zeta_{i}[\alpha_{2}]}\right) \label{pl_qi}%
\end{equation}
meets (\ref{form2}).
\end{Cor}

\section{Numerical results}
We consider the following  test functions defined on $I:=[0,1]$, to evaluate the effectiveness of the quasi-interpolation schemes introduced in this work. The two first functions are the 1D versions of Franke \cite{Franke} and Nielson \cite{Nielson} functions
\begin{align*}
f_{1}(x) & =\frac{3}{4}e^{-2(9x-2)^{2}}-\frac{1}{5}e^{-(9x-7)^{2}%
-(9x-4)^{2}}+\frac{1}{2}e^{-(9x-7)^{2}-\frac{1}{4}(9x-3)^{2}}+\frac{3}%
{4}e^{\frac{1}{10}(-9x-1)-\frac{1}{49}(9x+1)^{2}},\\
f_{2}(x) & =\frac{1}{2}x\cos^{4}\left(  4\left(  x^{2}+x-1\right)  \right),\quad f_{3}(x)  =x^{4}e^{-3 x^{2}}+\frac{1}{x^{6}+1}.
\end{align*}
Figure \ref{Test_functions} shows the plots of these test functions.
\begin{figure}[!h]
\centering
\includegraphics[scale=0.5]{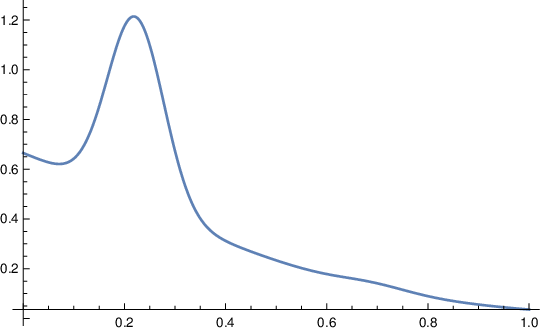}\quad \includegraphics[scale=0.5]{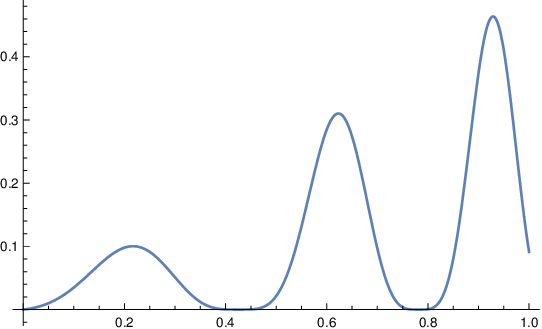}\quad \includegraphics[scale=0.5]{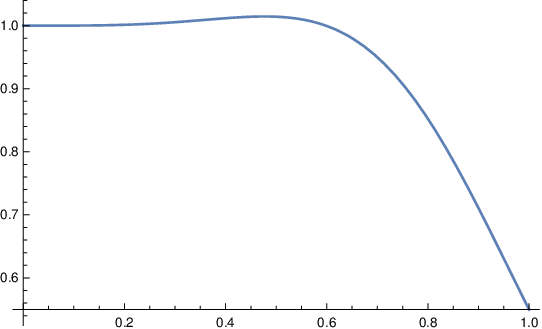}
\caption{Plots of the test functions, from left to right: $f_1$, $f_2$ and $f_3$.}\label{Test_functions}
\end{figure}

 Let us consider the bounded interval $I=[0, 1]$. The tests are carried out for a sequence of uniform partition $\tau_{n}$ of $I$ associated with the vertices $v_{i}:=i h$,  $i=0,\dots,n$, where $h:=\frac{1}{n}$. The inserted split points are the midpoints $\zeta_{i}:=\left(i+\frac{1}{2}\right) h$, $i=0,\dots,n-1$.


The quasi-interpolation error is estimated as
\[
\mathcal{E}_{n}:=\max_{0\leq\ell\leq200}\,|\mathcal{Q}f\left(  x_{\ell
}\right)  -f\left(  x_{\ell}\right) | 
\]
where $x_{\ell}, \ell=1, \ldots, 200$, are equally spaced points in $I$. Denote by $\mathcal{E}_{df,n}$ and  $\mathcal{E}_{dp,n}$ the estimated error $\mathcal{E}_{n}$ for the differential quasi-interpolant (\ref{df_qi}) and the discrete one based on polarization (\ref{pl_qi}), respectively. The Numerical Convergence Order (NCO) is given by the rate
\[
NCO:=\frac{\log\left(  \frac{\mathcal{E}_{n_{1}}}{\mathcal{E}_{n_{2}}%
}\right)  }{\log\left(  \frac{n_{2}}{n_{1}}\right)  }.
\]

In this work, an osculatory interpolation case has been treated, meaning that we do not have the same number of data at each knot. For that reason, and in order to be able to compare our numerical results with recent works, we consider the case of osculatory interpolation with three data at the vertices with an even index and four data at those with an odd index, i.e. $\varphi(2i)=3$, and $\varphi(2i+1)=4$.

Table \ref{tabdiff} shows the estimated errors relative to the differential quasi-interpolant (\ref{df_qi}) and the NCOs for  $f_{1}$ ,$f_{2}$ and $f_{3}$. 

\begin{table}[!htp]
\centering%
\begin{tabular}
[c]{|c|c|c|c|c|c|c|}\hline
$n$ & $\mathcal{E}_{df,n}(f_{1})$ & NCO & $\mathcal{E}_{df,n}(f_{2})$ & NCO& $\mathcal{E}_{df,n}(f_{3})$ & NCO \\\hline
\multicolumn{1}{|c|}{$16$} & \multicolumn{1}{|r|}{$1.4305\times10^{-3}$} &
$--$ & $1.84729\times10^{-3}$ &  $--$  & $1.78188\times10^{-6}$ & $--$ \\
\multicolumn{1}{|c|}{$32$} & \multicolumn{1}{|r|}{$6.11325\times10^{-5}$} & 
$4.54843$ & $1.46037\times10^{-4}$ & $3.661$ & $1.05981\times10^{-7}$ & $4.07153$ \\
\multicolumn{1}{|c|}{$64$} & \multicolumn{1}{|r|}{$6.48491\times10^{-6}$} &
$3.23678$ & $8.3046\times10^{-6}$ & $4.13628$ & $7.00157\times10^{-9}$ & $3.91998$  \\
\multicolumn{1}{|c|}{$128$} & \multicolumn{1}{|r|}{$3.66472\times10^{-7}$} &
$4.14531$ & $5.02361\times10^{-7}$ & $4.04712$ & $4.26038\times10^{-10}$ & $4.03862$ \\
\multicolumn{1}{|c|}{$256$} & \multicolumn{1}{|r|}{$2.48645\times10^{-8}$} &
$3.88154$ & $2.79162\times10^{-8}$ & $4.16955$ & $2.73356\times10^{-11}$ & $3.96213$ \\
\hline
\end{tabular}
\caption{Estimated errors of the differential Q.I (\ref{df_qi}) for the
functions $f_{1}$, $f_{2}$ and $f_{3}$  and NCOs with different values of $n$.}\label{tabdiff}
\end{table}
In Table \ref{tabpl1} we compare the results in Reference \cite{Jcam2023} for the test functions $f_1$ and $f_2$ with the errors provided by the discrete quasi-interpolant based on polarization formula.  Using up to four data in some vertices behaves better than the use of only three data in all the vertices, which explains the better results of the Table \ref{tabpl1}, as in reference \cite{Jcam2023} only three data are used in each vertex of the initial partition.

\begin{table}[!h]
\centering%
\begin{footnotesize}
\begin{tabular}
[c]{|c|c|c|c|c|c|c|c|c|}\hline
$n$ & $\mathcal{E}_{dp,n}(f_{1})$ & NCO & \text{Method in \cite{Jcam2023}} & NCO & $\mathcal{E}_{dp,n}(f_{2})$ & NCO & \text{Method in \cite{Jcam2023}} & NCO \\\hline
\multicolumn{1}{|c|}{$16$} & \multicolumn{1}{|r|}{$6.25902\times10^{-	4}$} &
$--$ &  $3.2851\times10^{-3}$ & $--$& $7.09872\times10^{-4}$ & $--$ & $8.3227\times10^{-3}$ & $--$  \\ 
\multicolumn{1}{|c|}{$32$} & \multicolumn{1}{|r|}{$3.10344\times10^{-5}$} &
$4.334$ & $3.8209\times10^{-4}$ & $3.10$ & $5.58036\times10^{-5}$ & $3.66913$ &  $5.1442\times10^{-4}$ & $4.01$ \\
\multicolumn{1}{|c|}{$64$} & \multicolumn{1}{|r|}{$2.42724\times10^{-6}$} &
$3.67648$ & $2.1478\times10^{-5}$ & $4.15$ & $3.8982\times10^{-6}$ & $3.83948$ & $2.9507\times10^{-5}$ & $4.12$ \\
\multicolumn{1}{|c|}{$128$} & \multicolumn{1}{|r|}{$1.76285\times10^{-7}$} &
$3.78334$ & $9.9300\times10^{-7}$ & $4.43$ &  $1.83669\times10^{-7}$ & $4.40763$& $1.8595\times10^{-6}$ & $3.98$\\
\multicolumn{1}{|c|}{$256$} & \multicolumn{1}{|r|}{$9.5628\times10^{-9}$} &
$4.20433$ & $7.5323\times10^{-8}$ & $3.72$&$1.29283\times10^{-8}$ & $3.8285$& $1.1592\times10^{-7}$ & $4.00$\\
\hline
\end{tabular}
\caption{Estimated errors of the quasi-interpolatant based on polarization and NCOs for the test functions $f_{1}$ and $f_{2}$ with different values of $n$.}\label{tabpl1}
\end{footnotesize}
\end{table}
In what follows,  we propose a comparison of the results obtained by the differential and discrete quasi-interpolants provided in this work with the results presented in \cite{Jcam2023,boujraf1,boujraf2,Rahouti}. To do that, the following test functions will be considered
\[
 g_{1}(x)= \sin(x),\quad 
g_{2}(x)= \frac{-1}{2}\left(\exp\left(\frac{x^{3}}{2}\right)-1\right) \cos\left(3 \pi x \right),\quad
\text{and}\quad
 g_{3}(x)=\exp\left(-3 x\right)\sin\left(\frac{\pi}{2}x\right).
\]

In Tables \ref{tabpl2}-\ref{tabpl4}, we list the resulting errors and NCOs for the approximation of functions $g_1$, $g_2$ and $g_3$ , respectively, by using the discrete quasi-interpolant based on polarization formulae (\ref{pl_qi}) presented here and those in references \cite{Jcam2023,boujraf1,boujraf2,Rahouti}. It is clear that the proposed scheme improves the results in those papers.
\begin{table}[!htp]
\centering%
\begin{tabular}
[c]{|c|c|c|c|c|c|c|}\hline
$n$ & $\mathcal{E}_{dp,n}(g_{1})$ & NCO &\text{Method in \cite{Jcam2023}} & NCO& \text{Method in \cite{boujraf1}} & NCO \\\hline
\multicolumn{1}{|c|}{$16$} & \multicolumn{1}{|r|}{$1.85742\times10^{-	9}$} &
$--$ & $2.1352\times10^{-8}$ &  $--$  & $3.25116\times10^{-7}$ & $--$ \\
\multicolumn{1}{|c|}{$32$} & \multicolumn{1}{|r|}{$1.16097\times10^{-10}$} &
$3.99989$ & $1.3627\times10^{-9}$ & $3.96$ & $2.13504\times10^{-8}$ & $3.92862$ \\
\multicolumn{1}{|c|}{$64$} & \multicolumn{1}{|r|}{$7.25608\times10^{-12}$} &
$4.00$ & $8.6021\times10^{-11}$ & $3.98$ & $1.36452\times10^{-9}$ & $3.96779$  \\
\multicolumn{1}{|c|}{$128$} & \multicolumn{1}{|r|}{$4.53415\times10^{-13}$} &
$4.00029$ & $5.4025\times10^{-12}$ & $3.99$ & $8.61907\times10^{-11}$ & $3.98472$ \\
\hline
\end{tabular}
\caption{Quasi-interpolation errors and NCOs for the test function $g_{1}$ with different values of $n$.}\label{tabpl2}
\end{table}

\begin{table}[!htp]
\centering%
\begin{tabular}
[c]{|c|c|c|c|c|c|c|}\hline
$n$ & $\mathcal{E}_{dp,n}(g_{2})$ & NCO &\text{Method in \cite{Jcam2023}} & NCO& \text{Method in \cite{boujraf2}} & NCO \\\hline
\multicolumn{1}{|c|}{$64$} & \multicolumn{1}{|r|}{$1.80915\times10^{-	8}$} &
$--$ & $2.1282\times10^{-7}$ & $--$  & $4.11124\times10^{-6}$ & $--$ \\
\multicolumn{1}{|c|}{$128$} & \multicolumn{1}{|r|}{$1.1413\times10^{-9}$} &
$3.98656$ & $1.3140\times10^{-8}$ & $4.01$ & $2.59629\times10^{-7}$ & $3.98505$ \\
\multicolumn{1}{|c|}{$256$} & \multicolumn{1}{|r|}{$7.13407\times10^{-11}$} &
$3.99981$ &  $7.3024\times10^{-10}$ & $4.16$ & $1.62327\times10^{-8}$ & $3.99948$  \\
\multicolumn{1}{|c|}{$512$} & \multicolumn{1}{|r|}{$4.46058\times10^{-12}$} &
$3.99942$ & $5.2902\times10^{-11}$ & $3.78$ & $1.01501\times10^{-9}$ & $3.99934$ \\
\hline
\end{tabular}
\caption{Quasi-interpolation errors and NCOs for the test function $g_{2}$ with different values of $n$.}\label{tabpl3}
\end{table}
\begin{table}[!htp]
\centering%
\begin{tabular}
[c]{|c|c|c|c|c|}\hline
$n$ & $\mathcal{E}_{dp,n}(g_{3})$ & NCO & \text{Method in \cite{Rahouti}} & NCO \\\hline
\multicolumn{1}{|c|}{$16$} & \multicolumn{1}{|r|}{$2.71919\times10^{-	7}$} &
$--$ & $3.2526\times10^{-7}$ & $--$\\
\multicolumn{1}{|c|}{$32$} & \multicolumn{1}{|r|} {$1.69909\times10^{-8}$} &
$4.00034$ &$2.2001\times10^{-8}$&$3.88593$ \\
\multicolumn{1}{|c|}{$64$} & \multicolumn{1}{|r|} {$1.06187\times10^{-9}$} &
$4.00008$ & $1.4290\times10^{-9}$ & $3.94444$\\
\multicolumn{1}{|c|}{$128$} & \multicolumn{1}{|r|}
{$6.6366\times10^{-11}$} &
$4.00002$ &  $9.1028\times10^{-11}$ & $3.97262$\\
\hline
\end{tabular}
\caption{Quasi-interpolation errors and NCOs for the test function $g_{3}$ with different values of $n$.}\label{tabpl4}
\end{table}

 In Table \ref{dataNum} we illustrate the number of data used for the construction of the quasi-interpolating splines used in this paper, and also for those used in the references \cite{Jcam2023,boujraf1,Rahouti}. Again, we consider the case of $\varphi (2i)=3$ and $\varphi (2i+1)=4$.

\begin{table}[!h]
\centering

\begin{tabular}{|c|c|c|c|c|c|c|}\hline
$ \overline{\tau}_n \setminus \mbox{Quasi-interpolant} $ & $\mathcal{E}_{df,n}$ & $\mathcal{E}_{dp,n}$ & \cite{Jcam2023} (cubic) & \cite{boujraf1} (cubic) & \cite{Rahouti} (cubic) \\\hline
$n$ & $\frac{7n}{2}$ & $9 n -1$ &  $ 3n-1 $  & $ n$  & $6n-1$ \\\hline
\end{tabular}
\caption{Data numbers used to construct the quasi-interpolation schemes proposed here, as well as those of the references \cite{Jcam2023,boujraf1,Rahouti}.}\label{dataNum}

\end{table}

Quasi-interpolation based on polarisation always needs more data than the other types. However, in the case of quasi-interpolant based on point values we can use the minimum number of data, i.e. we can use only $n$ data values.  Indeed, the operator defined in (\ref{QI_pv}) can be constructed from values at the partition points without the need for auxiliary values. It can be defined from values at the partition points surrounding the point under consideration.

\section{Conclusion}

In this work, we have proposed a normalized B-spline-like representation based on a two-point osculatory interpolation for a spline space defined on a refined partition. The splines considered are smooth and of low degrees.  The approach used to construct the basis functions is entirely geometric and yields splines with powerful properties such as non-negativity, the partition of unity, and compact support. We have also developed several quasi-interpolation operators based on blossoming and control polynomials. The provided quasi-interpolation schemes are tested numerically and demonstrate efficiency and performance with respect to some existing approaches.

\section*{Acknowledgements}
The authors wish to thank the anonymous referees for their very pertinent and useful comments which helped them to improve the original manuscript. The second author is a member of the research group GNCS of Italy and acknowledges the support of the MUR Excellence Department Project awarded to the Department of Mathematics, University of Rome Tor Vergata, CUP E83C23000330006. The  third author is a member of the research group FQM 191 Matem\'{a}tica Aplicada funded by the PAIDI programme of the Junta de Andaluc\'{i}a, Spain.

\end{document}